\documentclass{amsart}
\usepackage{amssymb}
\usepackage{graphicx,overpic} 
\usepackage{color}
\usepackage{tikz-cd}

\usepackage[hidelinks,pdftex]{hyperref}

\AtBeginDocument{%
   \def\MR#1{}
}

\numberwithin{equation}{section}
\theoremstyle{plain}
\newtheorem{theorem}[equation]{Theorem}
\newtheorem{corollary}[equation]{Corollary}

\newtheorem{proposition}[equation]{Proposition}

\newtheorem*{namedtheorem}{\theoremname}
\newcommand{\theoremname}{testing}
\newenvironment{named}[1]{\renewcommand{\theoremname}{#1}\begin{namedtheorem}}{\end{namedtheorem}}

\theoremstyle{definition}
\newtheorem{definition}[equation]{Definition}
\newtheorem{remark}[equation]{Remark}

\newtheorem{question}[equation]{Question}
\newtheorem{problem}[equation]{Problem}

\usepackage{microtype}

\usepackage[dvipsnames]{xcolor}
\definecolor{MyCyan}{HTML}{00F9DE}
\usepackage{enumitem} 
\setlist{nolistsep,leftmargin=*}
\usepackage{transparent} 
\usepackage{mathrsfs}
\usepackage[normalem]{ulem}

\usepackage{marginnote}
\long\def\@savemarbox#1#2{\global\setbox#1\vtop{\hsize\marginparwidth 
  \@parboxrestore\tiny\raggedright #2}}
\renewcommand{\setminus}{{\smallsetminus}}

\title{On the existence of universal links in three-manifolds}

\author[Francisco González]{Francisco González-Acuña} 
\address[]{Instituto de Matemáticas de la UNAM, Unidad Cuernavaca,  62210, Mexico} 
\email[]{ficomx@yahoo.com.mx}

\author[Araceli Guzmán]{Araceli Guzmán-Tristán} 
\address[]{Centro de Investigaci\'on en Matem\'aticas, GTO 36023, Mexico} 
\email[]{araceli.guzman@cimat.mx} 

\author[Jos\'{e} Rodr\'{\i}guez]{Jos\'{e} Andr\'{e}s Rodr\'{\i}guez Migueles}
\address[]{Centro de Investigaci\'on en Matem\'aticas, GTO 36023, Mexico}
\email[] {jose.migueles@cimat.mx} 

\author[Jesús Rodríguez]{Jesús Rodríguez-Viorato}
\address[]{Centro de Investigaci\'on en Matem\'aticas, GTO 36023, Mexico}
\email{jesusr@cimat.mx}

\date{\today}

\begin{document}

\begin{abstract}
We study the existence of branched coverings between closed $3$-manifolds, with emphasis on universal knots and links. We prove that the only closed $3$-manifolds that admit a universal link are spherical. Furthermore, we distinguish between universal links and \textit{complement universal} links and show that these notions do not coincide in general, by exhibiting infinitely many examples of complement universal links that are not universal.  Also, we  prove that there is no closed aspherical $3$-manifold, such that every closed,  aspherical  $3$-manifold is a branched covering over it. Finally, we characterize the closed $3$-manifolds admitting branching coverings from $P^3 \# P^3$, and deduce that there is no closed reducible $3$-manifold, such that every closed  reducible  $3$-manifold is a branched covering over it.

\end{abstract}

\maketitle

\section{Introduction}

Hilden proved in \cite{Hilde:BranchedCoverToS3} that every orientable $3$-manifold is a $3$-fold branched covering of the $3$-sphere $\mathbb{S}^3$. With this result as a starting point, Thurston proved in \cite{Thurston:UniversalLinks} the striking fact that there exists a link $L \subset \mathbb{S}^3$ such that every orientable $3$-manifold is a branched covering of $\mathbb{S}^3$ with branching set exactly $L$. Thurston called $L$ a universal link. The existence of universal links reflects geometric and group-theoretic constrains on the ambient manifold. In particular, it is natural to ask for a classification of closed $3$-manifolds that admit such links. In this paper, we answer this question:

\begin{named}{Theorem \ref{Thm:SphericalAreUniversal}}
A closed $3$-manifold admits universal knots and links, if and only if, it is spherical.
\end{named}

The necessity condition for the orientable case in Theorem \ref{Thm:SphericalAreUniversal}, comes from the fact that given any branched cover, between orientable closed $3$-manifolds $f:N\rightarrow M$, then $f_*:\pi_1(N)\rightarrow \pi_1(M)$ has an image of finite index in $\pi_1(M),$ which was first notice by Hopf in \cite[Satz VIII]{Hopf:ImageOfFiniteIndex} (see Proposition \ref{prop:hopf}). This means that if we have a branched cover from the $3$-sphere to $M,$ then the fundamental group of $M$ is finite, and by Geometrization (see \cite{Perelman:Geometrization}) $M$ is spherical. This can also be obtained using a related result in \cite[Lemma 2.4]{HughesKjuchukovaMiller:OpenBranchedCovers}. Then the goal of the proof of Theorem \ref{Thm:SphericalAreUniversal} is to show the existence of universal links in spherical manifolds and the necessity condition for the non-orientable case.

Furthermore, we notice that Thurston provided a more general definition of a universal link in \cite{Thurston:UniversalLinks}, by calling a link $L$ in an arbitrary orientable closed $3$-manifold $M$ \emph{complement universal} if there exists a link in any orientable closed $3$-manifold whose complement is virtually homeomorphic to $M\setminus L$ (see Section \ref{Sec:UniversalLink}). While every universal link is complement universal, the converse is not generally true, as we show: 

\begin{named}{Corollary \ref{Cor:UniversalLinkNotImplyUniversallyBranching}}
There exists a complement universal link in infinitely many $3$-manifolds which is not a universal link.
\end{named}

There is another homotopy obstruction, for the existence of branched covers between closed oriented manifolds, known by Hopf in \cite[Satz I]{Hopf:ImageOfFiniteIndex}, which is that the induced morphism in rational cohomology algebras is injective. In this setting, Berestein and Edmons proved in \cite{BeresteinEdmonds:BranchedCovers}, among other things, that an orientable $3$-manifold $M$ is a $3$-fold branched covering of $\mathbb{S}^2\times\mathbb{S}^1$ if and only if $H^1(M ) \neq 0$. Some time later, Núñez in \cite[Theorem 2.1]{Nuñez:UniversalLinksInS2XS1} gave an example of a link $L$ in $\mathbb{S}^2\times\mathbb{S}^1$  such that every closed, orientable $3$-manifold $N$ with $H^1(N ) \neq 0$ is a branched covering over  $\mathbb{S}^2\times\mathbb{S}^1$  with branching set exactly the link $L,$ and denoted these links as \textit{universal link for} $\mathbb{S}^2\times\mathbb{S}^1.$ 

In this paper we investigate a similar problem of finding a closed, aspherical  $3$-manifold $M$,  and  a link $L$ inside such that every closed, aspherical $3$-manifold $N$ is a branched covering over $M$  with branching set exactly the link $L.$ But we prove that such branched coverings can not even exist:

\begin{named}{Proposition \ref{Prop:NonExistanceAsphericalUniversalManifolds}}
There is no closed,  aspherical  $3$-manifold $M$ such that every closed,  aspherical  $3$-manifold $N$ is a branched covering over $M.$ 
\end{named}

Motivated by the previous result, we investigate also the problem of finding a closed, reducible  $3$-manifold $M$  such that every closed, reducible $3$-manifold $N$ is a branched covering over $M,$ and also proved that such branched coverings can not even exist:
\begin{named}{Corollary \ref{Coro:NonExistanceReducibleUniversalManifolds}}

There is no closed reducible $3$-manifold $M$ such that every closed, reducible  $3$-manifold $N$ is a branched covering over $M.$  
\end{named}

More information about the existence of branched coverings between manifolds can be found in \cite{HughesKjuchukovaMiller:OpenBranchedCovers}.

The paper is organized as follows. In section \ref{sec1}, we give some basic definitions about branched coverings, universality and transfer maps, followed by the proof of Theorem \ref{Thm:SphericalAreUniversal}. In section \ref{Sec:UniversalLink}, we prove the non-equivalence between the definitions of complement universal links and universal links. In section \ref{Sec:NonExistanceBranchedCovers}, we address the problem of finding $3$-manifolds that could have universal links within a restricted class of closed  $3$-manifolds, in Proposition \ref{Prop:NonExistanceAsphericalUniversalManifolds}, for aspherical manifolds and in Corollary \ref{Coro:NonExistanceReducibleUniversalManifolds}, for reducible manifolds. Finally, in section \ref{questions} we present some open problems and questions concerning complement universality and characterization on the existence of branched covers between some closed $3$-manifolds.

\subsection{Acknowledgements} We gratefully thank Dieter Kotschick for his useful advice and comments on an earlier preprint, mainly proving Proposition \ref{Prop:NonExistanceAsphericalUniversalManifolds} and recalling Hopf's paper \cite{Hopf:ImageOfFiniteIndex}. Also, we thank Mario Eudave and Marc Kegel for pointing us toward helpful references. Finally, the authors thank CIMAT for creating an attractive mathematical environment.

\section{Universal links in spherical $3$-manifolds}\label{sec1} 

We begin by defining branched coverings between 3-manifolds and universality.

\begin{definition}
An open proper map $\phi: M^3 \to N^3$ is called a \emph{branched covering} if there exists a link $L\subset N^3$ such that the restriction $\phi|: M^3 -\varphi^{-1}(L) \to N^3- L$ is a finite covering space. The link $L$ is called the \emph{branching set}.
\end{definition}

The following observation, originally due to Hopf, will be one of the main tools used in this work.

\begin{proposition}[Hopf {\cite[Satz VIII]{Hopf:ImageOfFiniteIndex}} ]\label{prop:hopf}
Let $f \colon M \to N$ be a branched covering between orientable, closed, connected $n$-manifolds. Then the image of the induced homomorphism $f_\sharp \colon \pi_1(M) \to \pi_1(N)$ is a subgroup of finite index; that is, $[\pi_1(N) : f_\sharp(\pi_1(M))] < \infty.$
\end{proposition}

\begin{proof}
Recall that a branched covering between closed oriented manifolds has non-zero degree; equivalently, the induced map $f_* \colon H_n(M) \to H_n(N)$ 
is non-zero.

Let $H = f_\#(\pi_1(M)) \le \pi_1(N)$, and let $\rho \colon \widetilde{N} \to N$ be the covering map corresponding to the subgroup $H$. By the lifting property of covering spaces, the map $f$ lifts to a continuous map $\widetilde{f} \colon M \to \widetilde{N},$ so that $f = \rho \circ \widetilde{f}$. Consequently, on top-dimensional homology we have $f_* = \rho_* \circ \widetilde{f}_* ,$ which shows that $f_*$ factors through $H_n(\widetilde{N})$.

If $\widetilde{N}$ were non-compact, then $H_n(\widetilde{N}) = 0$, implying that $f_* = 0$. This contradicts the fact that $f$ has non-zero degree. Therefore, $\widetilde{N}$ must be compact.

Since $\widetilde{N}$ is compact, the covering map $\rho \colon \widetilde{N} \to N$ is finite, which is equivalent to $H$ being a finite-index subgroup of $\pi_1(N)$. This completes the proof.
\end{proof}

\begin{definition}\label{Def:UniversallyBranchingLink}
We define a link $L$ in a closed $3$-manifold $M$ to be \emph{universal} if every orientable closed  $3$-manifold $N$ is a branched cover of $M$ over $L.$
\end{definition}

\begin{definition}
A \emph{spherical $3$-manifold} is the resultant quotient space $\mathbb{S}^3/\Gamma$, where $\Gamma$ is a finite subgroup of $SO(4)$ acting freely by rotations on $\mathbb{S}^3$. Special cases of these manifolds include lens spaces and prism manifolds.    
\end{definition} 

\textbf{Transfer maps.} Now we explore an homological obstruction for the existence of branched covering maps between the $3$-sphere and a given $3$-manifold, which we will use to argue the non-orientable case in Theorem \ref{Thm:SphericalAreUniversal}. 

The subject of transfer maps on covering maps is well-known. But its existence can be extended to non-zero degree maps or branched covers. In \cite{Smith:TransferRamifiedCoverings}, L. Smith defined a $d$-fold ramified covering map,  where branched coverings are a particular case, and proved a generalization of transfer maps for branched covers that includes non-orientable manifolds, as follows:

\begin{proposition}[Proposition 2.2 in \cite{Smith:TransferRamifiedCoverings}]\label{Prop:GeneralTransfer}
    Let $f : M^n \to N^n$ be a $d$-fold ramified covering  between $n$-manifolds, then there exist a \emph{transfer homomorphism of} 
\[ f_\# : H_k(M) \to H_k(N) \]
denoted as
\[ \tau_f: H_k(N) \to H_k(M), \]
such that 
\[ f_\# \circ\tau_f: H_k(N) \to H_k(N), \]
is multiplication by $d$.
\end{proposition}
\begin{corollary}\label{Coro:NonExistenceBranchedCoverBetti}
    Rational homology spheres do not branch cover a closed  $3$-manifold $M$ with $H_1(M)\neq 0.$\qed
\end{corollary}

\begin{corollary}\label{Coro:NonExistenceBranchedCoverNonOrientable}
      Rational homology spheres do not branch cover a non-orientable, closed  $3$-manifolds.
\end{corollary}

\begin{proof}
    Non-orientable closed $3$-manifolds $M$ have non-zero first Betti number \cite[Lemma 8.6]{BeresteinEdmonds:BranchedCovers}. Then by Corollary \ref{Coro:NonExistenceBranchedCoverBetti} a rational homology spheres can not branch cover a non-orientable, closed  $3$-manifold.
\end{proof}
\begin{theorem}\label{Thm:SphericalAreUniversal}
A closed $3$-manifold admits a universal knot or link if and only if it is spherical.
\end{theorem}

\begin{proof}
   By Theorem \cite[Lemma 2.4]{HughesKjuchukovaMiller:OpenBranchedCovers} and Proposition \ref{prop:hopf} we have the neccessity condition  when $M$ is an orientable closed $3$-manifold. For the non-orientable case, we have by Corollay \ref{Coro:NonExistenceBranchedCoverNonOrientable} that there is not branching cover from the $3$-sphere to $M.$ 
   
   Now we prove the existence of universal links. Notice that the same proof can be applied to knots.

Let $M= \mathbb{S}^3/\Gamma$ be a spherical manifold, where $\Gamma$ is a finite subgroup of $SO(4)$ acting freely on $\mathbb{S}^3$. Thus, the projection map $\pi\colon \mathbb{S}^3\to \mathbb{S}^3/\Gamma$ is a covering map.

Let $k\subset \mathbb{S}^3$ be a universal link. We can find an isotopy of $\mathbb{S}^3$ such that the image of $k$ under this isotopy, say $K$, is contained in an open ball $B$, and $\pi|_B\colon B\to \pi(B)$ is a homeomorphism. In this way, $\pi(K)=K'$ is a link in $M=\mathbb{S}^3/\Gamma$. We state that $K'$ is a universal link in $M$:

Let $M'$ be a closed, orientable $3$-manifold. Since the link $k\subset \mathbb{S}^3$ is universal and $K$ is equivalent to $k$, then $K\subset \mathbb{S}^3$ is universal, so $M'$ can be obtained as a covering of $\mathbb{S}^3$ branched along $K$, that is, there is an open, proper map $\varphi\colon M'\to \mathbb{S}^3$ such that the restriction $\varphi|\colon M'\backslash \varphi^{-1}(K)\to \mathbb{S}^3\backslash K$ is a finite covering space. Thus, the composition $\pi\circ\varphi\colon M'\to M$ is open and proper. Moreover, since $\pi|_B$ is a homeomorphism, $(\pi\circ\varphi)^{-1}(K')=\varphi^{-1}(\pi^{-1}(K'))=\varphi^{-1}(K)$ and the restriction $\pi\circ\varphi|\colon M'\backslash(\pi\circ\varphi)^{-1}(K')\to M\backslash K'$ is a finite covering space.

\end{proof}



\section{Complement universal links that are not universal}\label{Sec:UniversalLink}

In \cite{Thurston:UniversalLinks} Thurston gave a more general definition of a universal link, as follows:
\begin{definition}\label{Def:UniversalLink}
We define a link $L$ in a orientable closed  $3$-manifold $M$ to be  \emph{complement universal} if  every orientable closed  $3$-manifold $N$ contains a link $L'\subset N$ such that $N\setminus L'$ is homeomorphic to a finite sheeted covering space of $M\setminus L.$
\end{definition}
\begin{remark}
    To avoid confusion on the concepts. Our concept of universal, which is widely used in literature on branching covers, is called in  \cite{Thurston:UniversalLinks} under the name of \emph{universal branching} and our concept of complement universal is called in the same paper under the name of \emph{universal.}
\end{remark}

Notice that if $L$ is a universal link, it is a complement universal link, because the restriction of a branched covering map on the complement of the branching set is a finite covering map (see Def. \ref{Def:UniversallyBranchingLink}).

\begin{corollary}\label{Cor:UniversalLinkNotImplyUniversallyBranching}
There exists a complement universal link in infinitely many $3$-manifolds that is not universal.
\end{corollary}
\begin{proof}
    Consider $M$, a non-spherical $3$-manifold that is homeomorphic to a Dehn surgery over a hyperbolic universal link $L$ in a spherical $3$-manifold $S,$ e.g. three ortogonal geodesics in the $3$-torus, whose complement is homeomorphic to the Borromean ring complement in the $3$-sphere.
    By Thurston's hyperbolic Dehn filling theorem, this provides infinitely many $3$-manifolds $M.$ Moreover, for every $N$ orientable closed  $3$-manifold there is a branched cover $p:(N,p^{-1}(L))\rightarrow(S,L),$ which induces a finite cover from $N\setminus p^{-1}(L)$ to $S\setminus L,$ and the latter is homeomorphic  to the complement of the core $L'$ of the Dehn filled tori in $M.$  
    
   In conclusion, $L'\subset M$ is a complement universal link, that is not a universal link because $M$ is non-spherical by Theorem \ref{Thm:SphericalAreUniversal}. And by varying over the integral Dehn surgery on $L$ in  $S$, we obtain infinitely many non-homeomorphic $3$-manifolds $M.$
\end{proof}

In the following section, we restrict the notion of universality to a special family of closed $3$-manifolds.

\section{Non-existence of Branched covers between some $3$-manifolds}\label{Sec:NonExistanceBranchedCovers} 

Núñez gave in \cite{Nuñez:UniversalLinksInS2XS1} different notions of universality for links, one of which is the following:

\begin{definition}\label{Def:UniversalLinkS1xS2}
    A \textit{universal link for $\mathbb{S}^2\times\mathbb{S}^1$is} is a link $L$ in $\mathbb{S}^2\times\mathbb{S}^1,$ such that every closed, orientable $3$-manifold $N$ with $H^1(N ) \neq 0$ is a branched covering over  $\mathbb{S}^2\times\mathbb{S}^1$  with branching set exactly the link $L.$
\end{definition}

And gave in \cite[Theorem 2.1]{Nuñez:UniversalLinksInS2XS1} an example of a universal link for $\mathbb{S}^2\times\mathbb{S}^1.$ In the first part of this section, we prove that the analog notion of universality can not hold for the family of closed aspherical $3$-manifolds:

\begin{proposition}\label{Prop:NonExistanceAsphericalUniversalManifolds}
There is no closed aspherical $3$-manifold $M$ such that every closed, aspherical  $3$-manifold $N$ is a branched covering over $M.$    
\end{proposition}
\begin{proof}
First notice that if $M$ is non-orientable, then we have a hyperbolic rational homology sphere, e.g. the Fomenko–Matveev–Weeks manifold, which can not branch cover $M,$ because of Corollary \ref{Coro:NonExistenceBranchedCoverNonOrientable}.
If there is an orientable, closed, aspherical $3$-manifold $M$ such that every closed, aspherical $3$-manifold $N$ is a branched covering over $M.$ Then there is a non-zero degree map form $f:\mathbb{T}^3\rightarrow M.$ By Proposition \ref{prop:hopf}, we have that:
\begin{enumerate}
    \item $M$ is spherical, if  $f_\#(\pi_1(\mathbb{T}^3))=\{e\},$ or
    \item $M$ is virtually abelian, meaning $M$ has a finite covering with fundamental group $\mathbb{Z}^k,$ if   $f_\#(\pi_1(\mathbb{T}^3))\neq\{e\}.$
\end{enumerate}

The first case is impossible, because $M$ is aspherical. And in the second case, again by asphericity we have that $M$ is finitely covered by the 3-torus. However, there are closed aspherical $3$-manifolds, e.g. the ones with $Nil$-geometry, that are not branched covers over $M$ because of  \cite[Theorem 2 and Theorem 6]{KotschickNeofytidis:DominationByCircleBunles}.
\end{proof}
In the second part of this section, we prove in Theorem \ref{Thm:NonExistanceRedUniversalManifolds} that the only closed, reducible $3$-manifolds, to be branch covered by $P^3\# P^3,$ is itself. And conclude in Corollary \ref{Coro:NonExistanceReducibleUniversalManifolds} that the analog notion of universality in Definition \ref{Def:UniversalLinkS1xS2} can not hold for the family of closed, reducible $3$-manifolds. 



    $$\pi_1(M_1)\ast \pi_1(M_2)\rightarrow (\pi_1(M_1)/K_1)\times (\pi_1(M_2)/K_2)$$

\begin{theorem}\label{Thm:NonExistanceRedUniversalManifolds}
Let $M$ be a closed $3$-manifold. There exists a branched covering
\[
f \colon P^3 \# P^3 \longrightarrow M
\]
if and only if
\begin{enumerate}
    \item $M$ is a spherical manifold, or
    \item $M$ is homeomorphic to $P^3 \# P^3$.
\end{enumerate}
\end{theorem}

\begin{proof}
The converse is clearly true, because as we have shown in Theorem \ref{Thm:SphericalAreUniversal} all spherical $3$-manifolds are branch covered by any orientable, closed $3$-manifold. And when $M = P^3 \# P^3$ we have the identity map $P^3 \# P^3 \to P^3 \# P^3$ as a branched covering. 

Suppose there exists a branched covering
$f \colon P^3 \# P^3 \to M$. By Corollary \ref{Coro:NonExistenceBranchedCoverNonOrientable}, the manifold $M$ must be orientable. From now on, we assume that $M$ is orientable; in particular, $f$ is a map of non-zero degree.

Let $M = M_1 \# \cdots \# M_k$ be the (unique) prime decomposition of $M$, given by the Kneser--Milnor prime decomposition theorem. Each prime summand $M_i$ is either aspherical, $\mathbb{S}^1 \times \mathbb{S}^2$, or spherical.

Observe that no summand $M_i$ can be $\mathbb{S}^1 \times \mathbb{S}^2$. Indeed, if this were the case, collapsing all other summands and composing with $f$ would yield a non-zero degree map $P^3 \# P^3 \to \mathbb{S}^1 \times \mathbb{S}^2,$
which is impossible by Corollary \ref{Coro:NonExistenceBranchedCoverBetti}. Hence, each summand $M_i$ is either aspherical or spherical.

For each $i$, consider the map $f_i \colon P^3 \# P^3 \to M_i$ obtained by composing $f$ with the non-zero degree map $M \to M_i$ that collapses all connected summands other than $M_i$.

Suppose that $M_i$ is aspherical. Then $\pi_1(M_i)$ has no elements of order two (it does not have torsion elements in general). Since $\pi_1(P^3 \# P^3) \cong \mathbb{Z}_2 * \mathbb{Z}_2$,
it follows that the induced homomorphism $(f_i)_\# \colon \pi_1(P^3 \# P^3) \to \pi_1(M_i)$ is trivial, that is, $(f_i)_\#(\pi_1(P^3 \# P^3)) = \{e\}.$

By Proposition \ref{prop:hopf}, this implies that $M_i$ is spherical, contradicting the assumption that $M_i$ was aspherical. Therefore, all summands $M_i$ must be spherical.

Again by Proposition \ref{prop:hopf}, the image
$f_\#(\pi_1(P^3 \# P^3))$
is a subgroup of finite index in $\pi_1(M) = *_i \pi_1(M_i).$
By the Kurosh subgroup theorem \cite{kurosch}, the images of the generators of
\[
\pi_1(P^3 \# P^3) \cong \langle x, y \mid x^2 = y^2 = e \rangle
\]
must be conjugate into the free factors $\pi_1(M_i)$. That is, there exist elements $w_1, w_2 \in \pi_1(M)$ and indices $i,j$ such that
\[
w_1 f_\#(x) w_1^{-1} \in \pi_1(M_i)
\quad \text{and} \quad
w_2 f_\#(y) w_2^{-1} \in \pi_1(M_j).
\]

There are three possibilities: both, one, or neither of the elements $f_\#(x)$ and $f_\#(y)$ is trivial. If at least one of them is trivial, then the image $f_\#(\pi_1(P^3 \# P^3))$ is finite. By Proposition \ref{prop:hopf}, this implies that $M$ is spherical.

Thus, we are left with the case in which neither $f_\#(x)$ nor $f_\#(y)$ is trivial. In this situation, again by the Kurosh subgroup theorem, the image is isomorphic to $H = \mathbb{Z}_2 * \mathbb{Z}_2.$

By Proposition \ref{prop:hopf}, $H$ must be a finite-index subgroup of $G = *_i A_i$, where  $A_i = \pi_1(M_i)$ are finite groups.

Using the Euler characteristic of groups (see \cite{Toshiyuki:EulerCaracteristicOfGroups}), we have
\[
\chi(H) = [G : H] \, \chi(G).
\]
Since  $\chi(H) = \chi(\mathbb{Z}_2) + \chi(\mathbb{Z}_2) - 1 = 0,$ it follows that $\chi(G) = 0$. On the other hand,
\[
\chi(G) = \left( \sum_{i=1}^k \frac{1}{|A_i|} \right) - k + 1.
\]
This equality can hold only when $k = 2$ and $|A_1| = |A_2| = 2$.

Therefore, in this final case, the manifold $M$ has exactly two spherical summands with fundamental group $\mathbb{Z}_2$, and hence $M \cong P^3 \# P^3$. This completes the proof.

\end{proof}

\begin{corollary}\label{Coro:NonExistanceReducibleUniversalManifolds}
There is no closed reducible $3$-manifold $M$ such that every closed, reducible  $3$-manifold $N$ is a branched covering over $M.$  
\end{corollary}
\begin{proof}
    If there is an orientable, closed, reducible $3$-manifold $M$ such that every closed, reducible $3$-manifold $N$ is a branched covering over $M.$ Then by Theorem \ref{Thm:NonExistanceRedUniversalManifolds}, $M=P^3\# P^3.$
    
    Furthermore, if we choose $p\in\mathbb{Z}_{>2},$ then we would have a non-zero degree map from the connected sum of two lens spaces of order $p$ to $P^3\# P^3.$ But we have a contradiction, because due to Proposition \ref{prop:hopf}, $P^3\# P^3$ would be spherical.
\end{proof}

\section{Further comments}\label{questions}

As we point out in Section \ref{Sec:NonExistanceBranchedCovers}, Núñez gave in \cite{Nuñez:UniversalLinksInS2XS1}  an example of a  universal link for $\mathbb{S}^2\times\mathbb{S}^1$ (see Definition \ref{Def:UniversalLinkS1xS2}). It will be interesting to have a more complete result:

\begin{problem}\label{Problem:AllUniversalBetti_1NonZeroManifolds}
    Determine all the closed orientable $3$-manifolds $M$ with a link $L$ inside such that every closed, orientable $3$-manifold $N$ with $H^1(N ) \neq 0$  is a branched covering over  $M$  with branching set exactly the link $L.$
\end{problem}

Notice that for some  other topological constrains, e.g. being non-orientable, the analogue of Proposition \ref{Prop:NonExistanceAsphericalUniversalManifolds} and Corollary \ref{Coro:NonExistanceReducibleUniversalManifolds} is false, because Berestein and Edmonds proved in \cite{BeresteinEdmonds:BranchedCovers} that every closed non-orientable $3$-manifold is a branched covering of $P^2\times\mathbb{S}^1,$ but as pointed in \cite{Nuñez:UniversalLinksInS2XS1} there is no example of a universal link for $P^2\times\mathbb{S}^1. $ Then we have the following question:

\begin{problem}\label{Problem:ExistanceNonOrientableUniversalLink}
     Decide if there exist a link $L$ inside $P^2\times\mathbb{S}^1,$  such that every closed, non-orientable $3$-manifold $N$ is a branched covering over  $P^2\times\mathbb{S}^1$  with branching set exactly the link $L.$
\end{problem}

Following the opposite direction of Theorem \ref{Thm:NonExistanceRedUniversalManifolds}, we have that $\mathbb{S}^2\times\mathbb{S}^1$ is a $2$-folds cover  $P^3\# P^3,$ then every closed, orientable $3$-manifold $N$ with $H^1(N ) \neq 0$  is a branched cover over $P^3\# P^3.$ Then we could ask:

\begin{problem}\label{Problem:P3P3BranchedCovers}
      Identify which closed $3$-manifold branch cover $P^3\# P^3.$
\end{problem}

In the topic of complement universal link, we saw in Corollary \ref{Cor:UniversalLinkNotImplyUniversallyBranching}, the notions of universal link and universal branching link are not equivalent for non-spherical $3$-manifolds, but one could ask whether these notions coincide in the case of spherical $3$-manifolds.

\begin{question}
    Are there complement universal links in spherical $3$-manifolds that are not universal?
\end{question}

In the proof of Theorem~\ref{Thm:SphericalAreUniversal}, we construct universal knots in spherical 3-manifolds by taking a universal knot in $\mathbb{S}^3$ and embedding it into the manifold. 
This naturally raises the question of whether all universal knots in spherical manifolds arise in this way. 
More precisely, we can ask:

\begin{question}
Do there exist universal knots in spherical 3-manifolds that are not contained in a 3-ball?
\end{question}

\bibliography{references}
\bibliographystyle{amsplain}
\end{document}